\documentclass[a4paper,oneside,11pt]{article}

\usepackage{amsxtra,amssymb,amsmath,amsthm,amsbsy,enumerate}
\usepackage[all,2cell,v2]{xy}\UseTwocells\UseHalfTwocells
\usepackage{graphicx}

\usepackage{hyperref}

 \usepackage[dvipsnames]{xcolor}

\newtheorem{theorem}{Theorem}[section]

\newtheorem{conjecture}[theorem]{Conjecture}

\theoremstyle{definition}
\newtheorem{remark}[theorem]{Remark}
\newtheorem{corollary}[theorem]{Corollary}
\newtheorem{example}[theorem]{Example}

\newtheorem{proposition}[theorem]{Proposition}

\newtheoremstyle{named}{}{}{\itshape}{}{\bfseries}{.}{.5em}{\thmnote{#3}}
\theoremstyle{named}
\newtheorem*{namedtheorem}{}

\def\C{{\mathbb C}}

\renewcommand\P{\mathbb{P}}

\usepackage{tikz-cd}

\def\to{\rightarrow}

\def\to{\rightarrow}

\def\Z{\mathbb Z}

\title{\textbf{A note on Bondal's conjecture}}
\author{Dar\'io Mart\'in Aza}
\date{\today}

\begin{document}
{

\maketitle

\begin{abstract}
We prove that the connection vector fields associated to ample Poisson line bundles are not locally hamiltonian unless the Poisson structure is zero. We use this result to provide further evidence on Bondal's conjecture regarding the dimensions of the degeneracy loci of a holomorphic Fano Poisson manifold.
\end{abstract}




\section{Introduction}

Let $(X, \pi)$ be a holomorphic Poisson manifold, which means that $X$ is a holomorphic manifold and $\pi$ is a holomorphic bivector field such that $[\pi, \pi]=0$ for the Schouten bracket $[\; ,\;]$ of multivector fields. The study of holomorphic or algebraic Poisson geometry amounts to the study of Poisson brackets on the rings of holomorphic or algebraic functions on a holomorphic manifold or sometimes more singular spaces such as schemes or complex analytic varieties \cite{Polischuk} \cite{Pym}. To such a Poisson manifold we can assign a collection of Poisson subvarieties called degeneracy loci given by $$D_{2k}(\pi)=\{x\in X\; \vert \; \pi^{k+1}(x)=0 \}.$$ These subspaces provide $X$ with a stratification

$$D_0(\pi)\subseteq D_2(\pi) \subseteq D_4(\pi) \subseteq\dots D_{2r-2}(\pi)\subseteq X$$

\noindent where $r$ is the rank of the Poisson structure. In the year 1993, Bondal made the following conjecture involving these subspaces

\begin{conjecture} (\cite{bondal1993non}) If $(X, \pi)$ is a Fano Poisson complex manifold of rank $r$ then $D_{2k}$ has an irreducible component of dimension at least $2k+1$ for all $k\in \{0, 1, \dots, r-1\}.$
\end{conjecture}
 
Bondal's conjecture includes the following conjectures:
\begin{itemize}
     \item If $X$ is Fano then $\pi$ vanishes on a curve.
     \item If $X$ is Fano and $\pi$ of rank $2k$ then $D_{2k-2}$ has a component of dimension at least $2k-1.$ This result was proved in \cite{Polischuk} .    
\end{itemize}

In low dimensions the conjecture takes the following form 
\begin{itemize}
     \item In the case of dimension 3, where the rank of $\pi$ must be 2, the conjecture is reduced to proving the statement that $D_0$ contains a curve. This case was settled in \cite{Polischuk}.
     \item In the case of dimension 4 the conjecture is the pair of statements $D_0$ contains a curve and $D_2$ contains a subspace of dimension 3. This was proved in \cite{PoissonDegeneracy}.
     \item In the case of dimension 5 the conjecture is the same pair of statements as before: $D_0$ contains a curve and $D_2$ contains a subspace of dimension 3. This is still open.
\end{itemize}

In the year 1997, Polischuk published the seminal paper on algebraic Poisson structures \cite{Polischuk}, in which he tackled Bondal's conjecture, proving it in the case of dimension 3. In the work \cite{PoissonDegeneracy}, the authors approached Bondal's conjecture via the study Poisson modules. An invertible Poisson module is a line bundle $L$ together with a flat Poisson connection $\nabla$. To any invertible Poisson module there is associated a connection vector field and the connection vector field associated with the canonical module is the modular vector field which was initially studied in \cite{outerderivation} and \cite{weinstein}. The connection vector field is not a globally defined vector field but it is actually a global section of the quotient sheaf $\mathfrak{X}_{\text{Poiss}}\slash \mathfrak{X}_{\text{Ham}}$ between the sheaf of Poisson vector fields and that of hamiltonian vector fields. To each invertible Poisson module the authors of \cite{PoissonDegeneracy} assigned multiderivations supported on each degeneracy locus which could provide an explanation for the dimensions ocurring in Bondal's conjecture. These multiderivations were called residues. The k-th modular residue of a Poisson manifold $(X, \pi)$ is the globally defined section associated to the canonical module via this method. It would be interesting to relate the ampleness of an invertible Poisson module with nonvanishing properties of these residues in order to approach Bondal's conjecture. In this note we first prove the following

\begin{theorem}\label{teor} Let $(L, \nabla)$ be an ample Poisson line bundle on the complex Poisson analytic space $(X, \pi)$ and suppose that $\pi\neq0.$ Then the connection vector field of $(L, \nabla)$ defines a non trivial class in $H^0(X, \mathfrak{X}_{\text{Poiss}}\slash \mathfrak{X}_{\text{Ham}})$.
\end{theorem}

This means that the connection vector field assigned to an ample Poisson line bundle is not locally hamiltonian. Since Gualtieri-Pym's residues only depend on said class defined by the connection vector field of the line bundle, this is a necessary condition for the nonvanishing of Gualtieri-Pym's residues. In particular, if $(X, \pi)$ is a Fano Poisson manifold its modular vector field defines a nonvanishing section of the quotient sheaf $\mathfrak{X}_{\text{Poiss}}\slash \mathfrak{X}_{\text{Ham}}$. It is important that we prove that this result is valid on Poisson analytic spaces which are possibly singular, since we want to apply it to the degeneracy locus themselves of a Poisson manifold. In this way, we obtain our main result which gives some evidence in favour of Bondal's conjecture, more precisely proving that the dimensions appearing in Bondal's conjecture are satisfied by at least half of the degeneracy loci involved.

\begin{theorem} Given $(X, \pi)$ a complex Fano Poisson manifold such that $D_{2k-2}(\pi)\subsetneq D_{2k}(\pi),$ then either $D_{2k-2}(\pi)$ contains an irreducible component of dimension at least $2k-1$ or $D_{2k}(\pi)$ contains an irreducible component of dimension at least $2k+1$.
\end{theorem}

The paper is organized as follows: in section 2 we provide definitions involving complex Poisson manifolds and analytic spaces. In section 3 we review some facts concerning Poisson modules and Gualtieri-Pym's definition of the residues assigned to an invertible Poisson module. Our original results are discussed in section 4.

\paragraph{Acknowledgements.} This work was done while working on the author's PhD thesis at Universidad de Buenos Aires fully supported by CONICET and under the suppervision of Federico Quallbrunn and Mat\'ias del Hoyo. We thank Federico and Mat\'ias for their guidance and mentorship. We also thank Sebasti\'an Velazquez for carefully reading the first draft of this paper.



\section{Holomorphic Poisson Manifolds}

A \textbf{Poisson structure} on a complex manifold $X$ is a holomorphic bivector $\pi\in \Gamma (X, \Lambda ^2 T_X)$ such that  $[\pi, \pi]=0,$ where $[\; , \;]$ is the Schouten bracket of multivector fields. Whenever $(X, \pi)$ is a complex Poisson manifold $\mathcal{O}_X$ inherits a bracket which turns it into a sheaf of Poisson algebras. The bracket is given by $\{f,  g\}=\pi(df\wedge dg).$  We say that the \textbf{rank} of a Poisson structure $\pi$ is $r$ if $\pi^r\neq0$ but $\pi^{r+1}=0$. Holomorphic Poisson structures on a holomorphic manifold amount to two usual Poisson structures on the underlying differentiable manifold corresponding to the real and imaginary part of the bivector $\pi$ and a compatibility with the complex structure on the manifold: namely, if $\pi_I$ and $\pi_R$ are the imaginary and real parts of $\pi$ and $J$ the complex structure on $X,$ then $(\pi_I, J)$ has to form a Poisson Nijenhuis structure and $\pi_R^{\#}=J\circ \pi_I^{\#}.$ \\

We will also work with Poisson structures over complex analytic spaces which may not be smooth. Given a complex analytic space $X,$ we define the sheaf of k-multiderivations on $X$ as $$\mathfrak{X}_{X}^k=(\Omega_{X}^k)^*.$$ When $X$ is a complex manifold, $\Omega_{X}^k$ is a locally free sheaf and $(\Omega_{X}^k)^*\simeq \Lambda^k(\Omega_X^1)^*=\Lambda^k T_X$ but in general this may not be the case. The sheaf of multiderivations on a complex analytic space inherits a unique graded multilinear bracket extending the Lie bracket of derivations and we call this the Schouten bracket. A Poisson structure on a complex analytic space $X$ is a global biderivation $\pi\in \Gamma(X, \mathfrak{X}^2_{X})$ satisfying $[\pi, \pi]=0$ where $[, ]$ is the Schouten bracket of multiderivations. We remark that a bivector field induces a biderivation but in general a Poisson structure on a complex analytic space will not be induced by a global bivector field. \\

\begin{example}If $\mathfrak{g}$ is a complex Lie algebra, the dual space $\mathfrak{g}^{*}$ is a holomorphic Poisson manifold and is such that the bracket of linear functions is also a linear function. To see this, consider the inclusion $i:\mathfrak{g}\subseteq \mathcal{O}_{\mathfrak{g}^*}$, where $\mathfrak{g}$ is identified with the linear functions on $\mathfrak{g}^*$. Since $\mathfrak{g}$ generates $\mathcal{O}_{\mathfrak{g}^*}$ as a commutative algebra, there is a unique Poisson bracket $\{ \;,\, \}$ on $\mathcal{O}_{\mathfrak{g}^*}$ such that $\{i(x), i(y)\}=i([x, y])$ and this Poisson structure is called the Kostant-Soriau Poisson structure on $\mathfrak{g}^*$. Since the tangent sheaf to $\mathfrak{g}^*$ is free and generated by vector fields $\partial_{x_i}$ where $\{x_1, x_2, \dots , x_n\}$ is a basis for $\mathfrak{g},$ we have that the bivector field defining the Poisson structure is of the form $$\pi= \sum_{i<j} f_{ij} \partial_{x_i}\wedge \partial_{x_j}$$ and the functions $f_{ij}$ defining this bivector are linear functions of the form $f_{ij}= \sum_{k=1}^{n}c_{ij}^{k} i(x_k),$ where the $c_{ij}^{k}$ are the structure constants of the Lie algebra $\mathfrak{g}$.  
\end{example}


On a complex Poisson analytic space there are certain sheaves of vector fields which are distinguished: the Hamiltonian vector fields and the Poisson vector fields. Given a local holomorphic function $f$ on a complex Poisson analytic space $(X, \pi),$ we define a local holomorphic vector field $X_f = i_{df}(\pi)$ which we call its \textbf{hamiltonian vector field}. We denote by $\mathfrak{X}_\text{Ham}$ the sheaf of Hamiltonian vector fields. We call a local holomorphic function $f$ a \textbf{Casimir function} if its hamiltonian vector field is zero. We denote by $\text{Cas}_{X}$ the sheaf of holomorphic Casimir functions on $(X, \pi).$ A \textbf{Poisson vector field} on a complex Poisson manifold $(X, \pi)$ is a local holomorphic vector field $X\in \Gamma(U, T_X(U))$ such that $L_X (\pi)=0.$ We denote by $\mathfrak{X}_\text{Poiss}$ the sheaf of Poisson vector fields. Hamiltonian vector fields are always Poisson vector fields but not necessarily the opposite. \\

Both the sheaves of Poisson vector fields and of hamiltonian vector fields are not $\mathcal{O}_X$ modules sheaves: they are sheaves of $\C-$vector spaces. Hamiltonian vector fields form an integrable subsheaf of $T_X$ since $[X_f, X_g]=X_{\{f, g \}}.$  On a complex Poisson analytic space $(X, \pi),$ the image of the map $\pi^{\#}:\Omega^1_{X}\to T_X$ given by contraction with the Poisson tensor is a coherent subsheaf of $T_X$ generated as $\mathcal{O}_X$ module by the hamiltonian vector fields and is therefore an integrable distribution. It will be called the \textbf{symplectic foliation} of $(X, \pi)$ and written $\mathcal{F}_{\text{sym}}.$ The symplectic foliation satisfies that all its leaves are symplectic submanifolds (possibly singular) of $X$ and therefore are even dimensional: this includes the integral submanifolds passing through singular points of the foliation. 

\begin{example}Holomorphic symplectic manifolds are naturally holomorphic Poisson manifolds. If $(X, \pi)$ is a holomorphic Poisson manifold of dimension $2n$ and the contraction $\pi^{\#}: \Omega^1_X \to T_X$ is generically an isomorphism we say that $(X, \pi)$ is a generically symplectic Poisson manifold. If the anticanonical divisor $\text{Zeros}(\pi^n)$ is reduced as an analytic space, we call the Poisson structure \textbf{log-symplectic}. 
\end{example}

\begin{example}\label{Cerveau} On the projective space $\P^3,$ contraction with a section of the canonical bundle induces an isomorphism $H^0(X, \Lambda^2 T_X) \simeq H^0(X,\Omega^1_X\otimes \mathcal{O}(4))$. Under this identification, Poisson structures which do not vanish along a divisor correspond to twisted integrable 1-forms of degree $4$ which do not vanish along a divisor and those correspond to codimension 1 saturated foliations of degree $2$, which were classified in \cite{CerveauLinsNeto}. 
\end{example}

We now turn our attention to Poisson subvarieties. If $(X, \pi)$ is a Poisson analytic space, a \textbf{Poisson subspace} of $X$ is an analytic subspace $Y\subseteq X$ together with a Poisson structure $\sigma,$ such that the inclusion $i:Y\to X$ satisfies $\pi(df, dg)=\sigma (i^*(df), i^*(dg)).$ In case $Y\subseteq X$ has the structure of a Poisson subspace, this structure is unique and therefore we can say without ambiguity that $Y\subseteq X$ is a Poisson subspace.

\begin{proposition} Let $(X, \pi)$ be a Poisson analytic space and let $Y\subseteq X$ be an analytic subspace with ideal sheaf $I\subseteq \mathcal{O}_X.$ The following statements are equivalent:
\begin{itemize}
\item $Y$ admits the structure of a Poisson subspace. 
\item $I$ is a sheaf of Poisson ideals, meaning that $\{I, \mathcal{O}_X\}\subseteq I.$
\item Every local hamiltonian vector field is tangent to $Y.$
\end{itemize}
\end{proposition}
\begin{proof}
    A reference for the above result is \cite{PoissonDegeneracy}.
\end{proof}

Given $(X, \pi)$ a Poisson analytic space we call a subspace $Y\subseteq X$ a \textbf{strong Poisson subspace} if every local Poisson vector field is tangent to $Y.$ A strong Poisson subspace is always a Poisson subspace since hamiltonian vector fields are Poisson. We now define some important strong Poisson subspaces that can be considered inside any Poisson analytic space. Given a Poisson analytic space $(X, \pi)$ we define the $2k-th$ \textbf{degeneracy locus} of $\pi$ as  
$$D_{2k}=\text{Zeros}(\pi^{k+1}).$$
The degeneracy loci are invariant under the symmetries of the Poisson bivector and therefore are strong Poisson subspaces of $X.$ In this way the Poisson analytic space $X$ is endowed with a stratification by strong Poisson subspaces (possibly singular) 
$$D_0 \subseteq D_2 \subseteq D_4 \subseteq \dots \subseteq D_{2r}=X$$
where $2r$ is the rank of the bivector field $\pi.$ If $(X, \pi)$ is a Poisson manifold the highest non trivial degeneracy locus, $D_{2r-2},$ will be referred to as the singular set of the Poisson manifold. $D_{2r-2}$ coincides with the singular set of the symplectic foliation $\mathcal{F}_{\text{sym}}$. The degeneracy locus $D_{2k}$ admits also the following description
$$D_{2k}=\{x\in X \: | \text{   the symplectic leaf trough }x \text{ has dimension at most } 2k \}.$$


\section{Poisson modules}
On a Poisson manifold or analytic space there are distinguished vector bundles which serve as vector bundles with flat connections in Poisson geometry. Given a Poisson analytic space $(X, \pi)$ we will call a vector bundle $E$ over $X$ together with a flat Poisson connection $\nabla,$ a \textbf{Poisson module}. This means a pair $(E, \nabla)$ such that
$\nabla: E \to T_X \otimes E$ is a $\C-$linear morphism satisfying $\nabla(f s)=-X_f \otimes s + f\nabla (s) $ and such that the composition $\nabla \circ \nabla$ vanishes. This last condition is referred to as the \textbf{flatness} of the Poisson connection. Poisson modules induce an action of $\Omega^1_{X}$ on $E$ by regarding $\nabla_{\alpha}(s)=(\nabla(s))(\alpha).$ We will mainly focus our attention on Poisson invertible sheaves which we call \textbf{Poisson line bundles}. The set of Poisson line bundles forms a group under the usual tensor product. We call this group the \textbf{Picard-Poisson} group of $(X, \pi)$ and we denote it by $\text{PicPoiss}(X, \pi).$ The Picard Poisson group of $(X, \pi)$ can be identified with the first group hypercohomology of the complex of sheaves
$$0\longrightarrow\mathcal{O}_X^{\times}\longrightarrow T_X\longrightarrow \Lambda^2 T_X \longrightarrow ...$$
where the first map is given by $f\mapsto \frac{X_f}{f}.$ Given a Poisson line bundle $(L, \nabla)$ and a local trivialization $s$ of $L$ we get a vector field defined by 
$\nabla(s) = Z_{\text{conn}}\otimes s.$ This vector field is called the \textbf{connection vector field} of $(L, \nabla)$ associated to the trivialization $s$. Because of the flatness of $\nabla$ we have that $Z_{\text{conn}}$ is a Poisson vector field. Due to this fact, if $L$ is a Poisson line bundle over $(X, \pi)$ and $Y\subseteq X$ is a strong Poisson subspace then $L|_{Y}$ is a Poisson module over $(Y, \pi|_Y).$ \\

Given two different local trivializations $s, s'$ there is an holomorphic function $f$ such that $s=fs'$ in the intersection of both trivializing open sets. The connection vector field associated to $s$ and the one associated to $s'$ differ on $\frac{X_f}{f}$ in the intersection of both open sets. Because of this, even though there is no global connection vector field for a Poisson line bundle, there is a well defined global section of the quotient sheaf $\mathfrak{X}_{\text{Poiss}}\slash \mathfrak{X}_{\text{Ham}}$ between the sheaf of Poisson vector fields and that of hamiltonian vector fields, meaning that there is a covering by open subsets $U_i$ of $X$ such that on every open subset we have a Poisson vector field and on each intersection the chosen vector fields differ by a hamiltonian one. We will sometimes call this global section the connection vector field of the given Poisson line bundle. In this way the connection vector field defines a map
$$\partial_{\text{conn}}:\text{PicPoiss}(X, \pi)\to H^0(X, \mathfrak{X}_{\text{Poiss}}
\slash \mathfrak{X}_{\text{Ham}}).$$

\begin{example} The canonical line bundle $\omega_X$ of any Poisson manifold is a Poisson module. The connection can be described via the formula
$$\nabla_{\alpha}(s)=-\alpha\wedge d i_{\pi}(s)$$
for a given one form $\alpha$ and $s$ a section of $\omega_X.$ The connection vector field associated with this Poisson line bundle is called the \textbf{modular vector field} of $(X, \pi).$ The modular vector field is always locally hamiltonian around the regular points of the Poisson structure.
\end{example}

\begin{example}\label{ejemploGP} Let $X=\C^3$ with the holomorphic Poisson structure given by the bivector field $$\pi=c_{12}x_1 x_2 \partial_1 \wedge \partial_2+ c_{13}x_1 x_3 \partial_1 \wedge \partial_3+c_{23} x_2 x_3 \partial_2 \wedge\partial_3.$$
This Poisson structure is the restriction to an affine open subset of $\mathbb{P}^3$ of the Poisson structure induced by a foliation of type $L(1, 1, 1, 1)$ as in example \ref{Cerveau}. The symplectic foliation for this structure is generated by the hamiltonian vetor fields
\begin{align*}
\mathcal{F_{\text{sym}}}&=\langle i_{dx_1}(\pi), i_{dx_2}(\pi), i_{dx_3}(\pi)\rangle\\
&=\langle x_1(c_{12}x_2 \partial_2+c_{13}x_3\partial_3),x_2(-c_{12}x_1 \partial_1+c_{23}x_3\partial_3), x_3(-c_{13}x_1 \partial_1-c_{23}x_2\partial_2) \rangle.
\end{align*}
For generic values of $\{c_{12}, c_{13}, c_{23}\}$, the singular locus of $\pi$ is the union of the three coordinate axis. The modular vector field of this Poisson structure is a global Poisson vector field since $\omega_{\C^3}$ is trivial and can be computed to be $Z_{\text{mod}}=-(c_{12}+c_{13})x_1 \partial_1+(c_{12}-c_{23})x_2\partial_2+(c_{13}+c_{23})x_3\partial_3.$ Away from the singular locus, the modular vector field is tangent to the symplectic foliation as can be seen by the identity $Z_{\text{mod}}=\sum_{i=1}^{3} \frac{i_{dx_i}(\pi)}{x_i},$ however, on the singular locus $Z_{\text{mod}}$ does not vanish for generic values of $\{c_{12}, c_{13}, c_{23}\}$ even though all the hamiltonian vector fields do. 

\end{example}

\begin{example}
\begin{enumerate}
\item If $(X, 0)$ is a holomorphic manifold with a trivial Poisson structure then every line bundle in $X$ will have a Poisson module structure given by the trivial connection. The other Poisson module structures will be given by the choice of a holomorphic vector field in $X$ so that $\text{PicPoiss}(X,0)\simeq H^0(X, T_X)\times \text{Pic} (X).$

\item If $(X, \omega)$ is a simplectic manifold, then the line bundles that admit a Poisson module structure are those that admit a flat connection. Indeed, the symplectic structures provides an isomorphism $\omega : T_X\to T^{*}_X$ under which the Poisson module condition turns into $(\omega\otimes id) \circ \nabla: E\to T^{*}_X\otimes E$ being a flat connection. In this case, the Picard-Poisson group is given by the line bundles with vanishing Chern class.

\item If $(\mathbb{P}^n, \pi)$ is a holomorphic Poisson structure on $\mathbb{P}^n$ then every line bundle admits a flat Poisson connection. Indeed, since the canonical bundle of $\mathbb{P}^n$ admits a Poisson module structure we can endow $\mathcal{O}(1)$ with such a structure and therefore any line bundle. Two Poisson module structures on a given line bundle differ by a choice of a global Poisson vector field, giving that $\text{PicPoiss}(\mathbb{P}^n, \pi)\simeq H^0(\mathfrak{X}_{\text{Poiss}})\times \mathbb{Z}.$

\item If we endow $\mathbb{P}^1\times \mathbb{P}^1$ with a non zero Poisson structure, then the only line bundles that carry a Poisson module structure are those of the form $\mathcal{O}(n, n)$ for $n\in \mathbb{Z}$ and therefore $\text{PicPoiss}(\mathbb{P}^1\times \mathbb{P}^1, \pi)\simeq H^0(\mathfrak{X}_{\text{Poiss}})\times \mathbb{Z}.$ In this case, the forgetful map $\text{PicPoiss}\to \text{Pic}$ is not surjective. \\
\end{enumerate}
\end{example}

\begin{remark} The connection vector field map $\partial_{\text{conn}}:\text{PicPoiss}(X, \pi)\to H^0(X, \mathfrak{X}_{\text{Poiss}}\slash \mathfrak{X}_{\text{Ham}})$ can also be obtained as the morphism induced in cohomology by the following natural map between complexes of sheaves

$$(0\to \mathcal{O}_{X}^{\times}\to T_X \to \Lambda^2 T_X \to \Lambda^3 T_X \to \dots )\to (0\to 0 \to T_X \slash \mathfrak{X}_{\text{Ham}} \to \Lambda^2 T_X\to \Lambda^3 T_X \to \dots).$$

\noindent The kernel of this map of complexes is given by the complex of sheaves $$(0\to \mathcal{O}_{X}^{\times}\to \mathfrak{X}_{\text{Ham}} \to 0 )$$
which is a resolution of the sheaf $\text{Cas}^{\times}_{X}$ of nonvanishing Casimir funcions on $X$. Therefore, the map $\partial_{\text{conn}}$ sits in a long exact sequence in cohomology which is of the form

$$0\to H^1(X, \text{Cas}^{\times}_{X}) \to \text{PicPoiss}(X, \pi)\to H^0(X, \mathfrak{X}_{\text{Poiss}}\slash \mathfrak{X}_{\text{Ham}}) \to H^2(X, \text{Cas}^{\times}_{X}) \to \dots $$

\end{remark}

\noindent We now recall the definition of \textbf{residues} for invertible Poisson modules introduced in the work \cite{PoissonDegeneracy}. These residues are defined to be multiderivations supported on each degeneracy locus induced by the connection vector field of a Poisson line bundle. Given a Poisson line bundle $(L, \nabla)$, its \textbf{k-th Gualtieri-Pym residue} is 
$$\text{Res}_{k}^{GP}(L, \nabla)=\left(Z_{\text{conn}}\wedge \pi^{k} \right)|_{D_{2k}}\in \Gamma(D_{2k}, \Lambda^{2k+1}T_{D_{2k}})$$
where $Z_{\text{conn}}$ is the connection vector field of $(L, \nabla).$ The definition of each residue is independent of the trivialization chosen to provide the connection vector field, since the wedge product $\pi^{k}\wedge X_f$ vanishes in $D_{2k}$ for every hamiltonian vector field $X_f.$  Gualtieri-Pym's residues are actually constructed only using the connection vector field of the line bundle so that the $k-th$ residue defines a map 
$$\Gamma(X, \mathfrak{X}_{\text{Poiss}}\slash \mathfrak{X}_{\text{Ham}})\to \Gamma(D_{2k}, \Lambda^{2k+1}T_{D_{2k}}).$$ 
The residues associated to the canonical Poisson line bundle are called \textbf{modular residues} of $(X, \pi).$ Gualtieri-Pym's residue of a Poisson line bundle vanishes at the points where the connection vector field lies in the hamiltonian distribution and therefore the vanishing of all modular residues implies that the modular vector field is locally hamiltonian.
 
\begin{example}  Let $X=\C^3$ with the holomorphic Poisson structure given by the bivector field $$\pi=c_{12}x_1 x_2 \partial_1 \wedge \partial_2+ c_{13} x_1 x_3 \partial_1 \wedge \partial_3+c_{23} x_2 x_3 \partial_2 \wedge\partial_3,$$ which we have already studied in example \ref{ejemploGP}. For general values of $c_{12}, c_{13}$ and $c_{23}$ the singular locus is the union of the three lines $D_0=\{x_1=x_2=0\}\cup\{x_1=x_3=0\}\cup \{x_2=x_3=0\}=L_{12}\cup L_{13}\cup L_{23}.$ Using the computations from example \ref{ejemploGP} we see that
$$\text{Res}_{0}^{GP}(\omega_{\C^3}, \nabla_{\text{mod}})|{L_{12}}=(c_{13}+c_{23})x_3\partial_3$$
$$\text{Res}_{0}^{GP}(\omega_{\C^3}, \nabla_{\text{mod}})|{L_{13}}=(c_{12}-c_{23})x_2\partial_2$$
$$\text{Res}_{0}^{GP}(\omega_{\C^3}, \nabla_{\text{mod}})|{L_{23}}=-(c_{12}+c_{13})x_1\partial_1.$$
while $\text{Res}_{1}^{GP}(\omega_{\C^3}, \nabla_{\text{mod}})$ vanishes identically.
\end{example}


\section{Ample Poisson modules}
We begin by some remarks involving the sheaves of Casimir functions, hamiltonian and Poisson vector fields on a Poisson analytic space. 

\begin{remark}\label{remark1} There is a natural short exact sequence of sheaves 
$$0\to \mathfrak{X}_{\text{Ham}}\to \mathfrak{X}_{\text{Poiss}}\to \mathfrak{X}_{\text{Poiss}}\slash \mathfrak{X}_{\text{Ham}}\to 0$$
which induces a morphism 
$$\delta: H^0(X, \mathfrak{X}_{\text{Poiss}}\slash \mathfrak{X}_{\text{Ham}})\to H^1(X, \mathfrak{X}_{\text{Ham}}).$$
There is also another short exact sequence of sheaves on any Poisson analytic space encoded by the definition of Casimir function which is given by
$$0\to \text{Cas}_X\to \mathcal{O}_X\to\mathfrak{X}_{\text{Ham}}\to 0$$
this induces a map $\xi:H^1(X, \mathfrak{X}_{\text{Ham}})\to H^2(X, \text{Cas}_X).$ The composition of both morphisms provides us with a map $$\eta=\xi\circ \delta:H^0(X, \mathfrak{X}_{\text{Poiss}}\slash\mathfrak{X}_{\text{Ham}})\to H^2(X, \text{Cas}_X).$$
\end{remark}

\begin{proposition} \label{diagrama}The following diagram is commutative
$$\begin{tikzcd}
\text{PicPoiss}(X, \pi) \arrow{r}{c_1}\arrow{d}{\partial_{\text{conn}}} & H^2(X, \mathbb{Z}) \arrow{d}{i_*} \\
H^0(X, \mathfrak{X}_{\text{Poiss}}\slash \mathfrak{X}_{\text{Ham}}) \arrow{r}{\eta} & H^2(X, \text{Cas}_X) 
\end{tikzcd}$$
\end{proposition}
\begin{proof} We have to verify that $\eta(\partial_{\text{conn}}(L, \nabla))=i_*(c_1(L)).$ Choose a family of open subsets $(U_i)$ with trivializations $s_i$ of $L$ and a Poisson vector field $Z_i$ on each of them such that $\nabla(s_i)=Z_i\otimes s_i.$ Then, the map $\partial_{\text{conn}}(L, \nabla)=(U_i, Z_i)$ seen as a Cech cocycle. We also have that $Z_i - Z_j = \frac{X_{g_{ij}}}{g_{ij}}$ where $(g_{ij})$ is the cocycle representing $L$ in $H^1(X, \mathcal{O}^{\times}_X).$ We then have that $\delta (\partial_{\text{conn}})(L, \nabla)= \frac{X_{g_{ij}}}{g_{ij}}$ and that $$\eta\circ \partial_{\text{conn}}(L, \nabla)=  (\text{log}(g_{ij})+\text{log}(g_{jk})+\text{log}(g_{ki}))_{ijk}$$
seen as a 2-cocycle of Casimir functions associated with the above covering. Since this is exactly the same cocycle that represents $c_1(L)$ we have that the diagram commutes. 
\end{proof}

\begin{corollary} If $(L, \nabla)$ is a Poisson module such that $i_*(c_1(L))\neq 0$ then the connection vector field associated to $(L, \nabla)$ is not locally hamiltonian.
\end{corollary}

\begin{remark} The map $i_*$ above fits in a long exact sequence induced by the exponential short exact sequence 
$$0\to \Z_X \to \text{Cas}_X\to \text{Cas}^{\times}_X\to 0.$$
Therefore, the kernel of $i_*$ consists of the Chern classes of line bundles $L$ that can be represented by a cocycle of Casimir functions. 
\end{remark}

\begin{remark} The connection vector field of a Poisson module together with the map $\eta$ defined on remark \ref{remark1} provides us with a map $\eta\circ \partial_{\text{conn}}: \text{PicPoiss}(X, \pi)\to H^2(X, \text{Cas}_{X})$ which is part of the following diagram 
$$\begin{tikzcd}
H^1 (X, \text{Cas}^{\times}_X)\arrow{d}\arrow{r}& H^1 (X, \text{Cas}^{\times}_X)\arrow{d}{c_1} &  \\
\text{PicPoiss}(X, \pi) \arrow{r}{c_1}\arrow{d}{\partial_{\text{conn}}} & H^2(X, \mathbb{Z}) \arrow{d}{i_*}\\
H^0(\mathfrak{X}_{\text{Poiss}}\slash \text{Ham}) \arrow{r}{\eta}\arrow{d} & H^2(X, \text{Cas}_{X})\arrow{d} &  \\
H^2 (X, \text{Cas}^{\times}_X)\arrow{r}{} & H^2 (X, \text{Cas}^{\times}_X) & \end{tikzcd}$$

\end{remark}

\noindent The following theorem proves that an ample Poisson line bundle has a non zero connection vector field.
\begin{namedtheorem}[Theorem 1.2] Let $X$ be a projective complex analytic space with a Poisson structure $\pi$. If $(L, \nabla)$ is an ample Poisson line bundle on the Poisson analytic space $(X, \pi)$, with $\pi\neq0,$ then the connection vector field of $(L, \nabla)$ defines a non trivial class in $H^0(X, \mathfrak{X}_{\text{Poiss}}\slash \mathfrak{X}_{\text{Ham}})$.
\end{namedtheorem}
\begin{proof}
Since the kernel of $i_*$ is isomorphic to the image of the group $H^1(X, \text{Cas}_{X}^{\times})$ it is enough to prove that ample line bundles cannot be represented by cocycles of Casimir non vanishing functions. We can assume $L$ to be very ample. Suppose that $L$ is given by a cocyle of the form $(g_{ij})_{ij}\in H^1(X, \text{Cas}_{X}^{\times})$ and consider the embedding $j:X\to \mathbb{P}^{m}_{\C}$ provided by the very ample line bundle $L.$ Then, $L\simeq j^*(\mathcal{O}(1))$ and therefore $g_{ij}=(\frac{x_i}{x_j})|_{X}$ for $(x_0:\dots : x_m)$ homogeneous coordinates on $\mathbb{P}^{m}_\C.$ Consider the open affine subset given by $U_j=\{x_j \neq 0\}.$ Since the functions $g_{ij}$ are Casimir functions for all $i$ we have that the affine coordinates on this chart given by $g_{ij}=z_i=\frac{x_i}{x_j}$ are all functions which restricted to $X$ give Casimir functions. We pick a point $p=(p_0, p_1 ,  \dots p_m)\in X\cap U_j$ and the ideal the ideal $I=(z_i-p_i)_{\{i\neq j\}}.$ The ideal $I$ is a Poisson ideal since it is generated by Casimir functions and therefore the point $p\in X$ is a Poisson subvariety of $X.$ Because of this, every hamiltonian vector field must vanish on $p.$ Since $p$ is an arbitrary point, we get that the Poisson bracket vanishes over $X.$
\end{proof}

This means that the connection vector field associated to an ample Poisson line bundle cannot be locally hamiltonian. It is interesting to study further properties of this connection vector field related to its non-vanishing along submanifolds of $X.$ In the following corollary we point out an observation on this direction.

\begin{corollary} Let $X$ be a projective complex analytic space with a Poisson structure $\pi$. If $(L, \nabla)$ is an ample Poisson line bundle on the Poisson analytic space $(X, \pi)$ and $V\subseteq X$ is a strong Poisson subspace such that $D_0(\pi)\subsetneq V,$ then the class in $H^0(X, \mathfrak{X}_{\text{Poiss}}\slash \mathfrak{X}_{\text{Ham}})$ defined by the connection vector field of $(L, \nabla)$ does not vanish along $V.$
\end{corollary}

\begin{proof} Since $V$ is a strong Poisson subspace we have that $(L, \nabla)$ restricts to a Poisson module over $(V, \pi|_V).$ Since $D_0(\pi)\subsetneq V,$ we have that $\pi|_V\neq 0$ and since $L|_V$ is an ample line bundle we can use theorem 1.2 on $V$ and obtain that the connection vector field of $L|_V$ defines a non trivial class on $H^0(V, \mathfrak{X}_{\text{Poiss}}\slash \mathfrak{X}_{\text{Ham}}),$ which is what we wanted to prove. 
\end{proof}

It would be interesting to know if the connection vector field of an ample Poisson line bundle can vanish along $D_0(\pi).$ If this was not possible the modular vector field of a Fano Poisson manifold would always be a non zero vector field tangent to $D_0(\pi)$ proving that $D_0(\pi)$ contains a curve. There are known examples where the degeneracy locus $D_0(\pi)$ has a connected component which is composed of an isolated point (for instance, the Poisson structures on $\mathbb{P}^3$ associated with logarithmic foliations of type L(1, 1, 2)), since these isolated points are strong Poisson submanifolds it becomes aparent that the connection vector field of an ample Poisson line bundle can indeed vanish along certain strong Poisson subspaces contained in $D_0(\pi)$. 

\begin{corollary} Let $X$ be a projective complex analytic space with a Poisson structure $\pi$ of constant rank $k\neq 0$. If $(L, \nabla)$ is an ample Poisson line bundle on the Poisson analytic space $(X, \pi)$, then $\text{dim}(X)\geq 2k+1$.  
\end{corollary}

\begin{proof} Since $\pi $ has rank $k$ in $X$ we have that the Poisson structure on $X$ does not vanish and then we can apply theorem 1.2 to $X.$ This shows that $i_*(c_1(L))\neq 0$ since $L$ is an ample line bundle and therefore its connection vector field is a non trivial class in $[Z]\in H^0(X, \mathfrak{X}_{\text{Poiss}}\slash \mathfrak{X}_{\text{Ham}}).$ We finally consider Gualtieri-Pym's modular residue $R=Z\wedge \pi^k$ which is a multiderivation in $X$ of degree $2k+1.$ Since $Z$ is not locally hamiltonian there is at least one point where $R$ is non zero and therefore $\text{dim}(X)\geq 2k+1.$
\end{proof}

\noindent Since the argument in theorem 1.2 is local, a similar result is valid even when considering a quasi-projective non complete variety. 

\begin{corollary} \label{coroco}Let $U$ be a quasi projective complex analytic space with a Poisson structure $\pi$ of constant rank $k\neq 0$. If $(L, \nabla)$ is a Poisson line bundle on $U$ such that $L$ is the pullback of an ample line bundle on a projective space and $c_1(L)\neq 0,$ then $\text{dim}(U) \geq 2k+1$.
\end{corollary}


\noindent We are now in position to prove our main theorem.

\begin{namedtheorem}[Theorem 1.3] Given $(X, \pi)$ a complex Fano Poisson manifold such that $D_{2k-2}(\pi)\subsetneq D_{2k}(\pi),$ then either $D_{2k-2}(\pi)$ contains an irreducible component of dimension at least ${2k-1}$ or $D_{2k}(\pi)$ contains an irreducible component of dimension at least $2k+1$.
\end{namedtheorem}
\begin{proof}
    Since the inclusion is strict, there exists a point $p\in D_{2k}(\pi)-D_{2k-2}(\pi).$ We have then that $\pi^k(p)\neq 0$ and therefore $D_{2k}(\pi)$ has an irreducible component passing through $p$ of dimension at least $2k.$ We call that component $V.$ Since $V$ is a strong Poisson subspace of $X$ the anticanonical bundle $(\omega_X)^{-1}$ of $X$ restricts to $V$ producing an ample Poisson line bundle over $V$ which we call $L.$ Let $F=D_{2k-2}(\pi)\cap V$ and $U=V-F.$ We know that $c_1(L)\in H^2(V, \C)$ is not zero and we have the following exact sequence of relative cohomology 
    $$\dots\to H^2(V,U, \C)\to H^2(V, \C)\to H^2(U, \C)\to \dots $$
    but then either the class of $c_1(L|_U)\neq 0$ or there exists a nontrivial class on $H^2(V, U, \C)$. There is a map  $H^2(V, U, \C)\to H_{2n-2}(F, \C),$ where $n$ is the complex dimension of $V$ which is an isomorphism if $V$ is smooth, given by the cap product $-\cap [V].$ Suppose that $\xi\in H^2(V, U, \C)$ is such that $i_*(\xi)=c_1(L)$ and consider the class $\xi\cap [V]\in H_{2n-2}(F, \C).$ We have that $i_*(\xi\cap [V])=c_1(L)\cap [V]\neq 0$ and therefore $H_{2n-2}(F, \C)\neq 0,$ which implies that $\text{dim}(F)\geq n-1\geq 2k-1.$\\

Now suppose that $c_1(L|_U)\neq 0.$ Since $\pi $ has rank $k$ in $U$ we have that the Poisson structure on $U$ does not vanish and then we can apply corollary \ref{coroco} to $U$ and therefore $\text{dim}(U)\geq 2k+1,$ showing that $\text{dim}(V)\geq 2k+1.$
\end{proof}

\noindent The above theorem shows that at least half of the subspaces defining degeneracy locus of a Fano Poisson variety satisfy the dimension bounds on Bondal's conjecture.



{\small

\begin{thebibliography}{xxx}


\bibitem[Bon93]{bondal1993non}
Bondal, Alexey;
Non-commutative Deformations and Poisson Brackets on Projective Spaces;
Max-Planck-Institut f{\"u}r Mathematik, (1993).
%

\bibitem[BZ99]{outerderivation}
Brylinski, Jean-Luc; Zuckerman, Gregg; 
The outer derivation of a complex Poisson manifold;
J. Reine Angew. Math. 506 (1999), 181--189.

\bibitem[CLN96]{CerveauLinsNeto}
Cerveau, Dominique and Lins Neto, Alcides;
Irreducible Components of the Space of Holomorphic Foliations of Degree Two in CP(n), n $\geq$ 3;
Annals of Mathematics, Second Series, Vol. 143, No. 3 (May, 1996), pp. 577-612.

\bibitem[GP12]{PoissonDegeneracy}
Gualtieri, Marco and Pym, Brent;
Poisson modules and degeneracy loci;
Proc. Lond. Math. Soc. (3) 107 (2013), no. 3, 627--654.

\bibitem[Pol97]{Polischuk}
Polischuk, Alexander;
Algebraic geometry of Poisson brackets;
Journal of Mathematical Sciences (New York), Vol. 84, No. 5 (1997), pp. 1413--1444.

\bibitem[Pym18]{Pym}
Pym, Brent;
Constructions and classifications of projective Poisson varieties;
Letters in Mathematical Physics , Vol. 108, No. 3 (2018), pp. 573--632.

\bibitem[Wei97]{weinstein}
Weinstein, Alan;
The modular automorphism group of a Poisson manifold;
Journal of Geometry and Physics, Vol 23, No. 3-4 (1997), pp. 379--394.

\end{thebibliography}
}
\end{document}